\documentclass{amsart}
\usepackage{amssymb,latexsym, amsmath, amscd, array, graphicx}
\usepackage {tikz-cd}

\swapnumbers
\numberwithin{equation}{section}

\theoremstyle{plain}

\newtheorem{thm}{Theorem}[section]
\newtheorem{cor}[thm]{Corollary}

\newtheorem{lem}[thm]{Lemma}

\newtheorem{prop}[thm]{Proposition}

\theoremstyle{definition}

\newtheorem{remark[thm]}{Remark}

\def\Hom{\protect\operatorname{Hom}}

\def\Im{\protect\operatorname{Im}}

\def\Ind{\protect\operatorname{Ind}}

\def\Ker{\protect\operatorname{Ker}}

\def\cat{\protect\operatorname{cat}}

\def\hd{\protect\operatorname{hd}}
\def\cd{\protect\operatorname{cd}}
\def\gd{\protect\operatorname{gd}}

\newcommand \pa[2]{\frac{\partial #1}{\partial #2}}

\def\scr{\mathcal}
\def\A{{\scr A}}
\def\B{{\scr B}}

\def\Z{{\mathbb Z}}
\def\Q{{\mathbb Q}}
\def\R{{\mathbb R}}

\def\k{{\mathbf k}}
\def\I{{\mathbf I}}

\def\B{{\mathcal B}}

\def\1{\hbox{\rm\rlap {1}\hskip.03in{\rom I}}}
\def\Bbbone{{\rm1\mathchoice{\kern-0.25em}{\kern-0.25em}
{\kern-0.2em}{\kern-0.2em}I}}

\def\pa{\partial}

\def\wt{\widetilde}

\long\def\forget#1\forgotten{} %

\begin{document}

\title[On cohomological dimension of homomorphisms  ]
{On cohomological dimension of homomorphisms  }
\author[A. De Saha]{Aditya De Saha}
\author[A.~Dranishnikov]
{Alexander Dranishnikov}

\thanks{The  second author was supported by Simons Foundation}

\address{A. De Saha, A. Dranishnikov, Department of Mathematics, University
of Florida, 358 Little Hall, Gainesville, FL 32611-8105, USA}
\email{dranish@ufl.edu}
\email{a.desaha@ufl.edu}


\begin{abstract} 
The (co)homological dimension of homomorphism $\phi:G\to H$ is the maximal number $k$ such that the induced homomorphism is nonzero for some $H$-module.
The following theorems are proven.
\begin{thm}
For every homomorphism $\phi:G\to H$ of a geometrically finite group $G$ the homological dimension of $\phi$ equals the cohomological dimension, $\hd(\phi)= \cd(\phi)$.
\end{thm}
\begin{thm}
For every homomorphism $\phi:G\to H$ of geometrically finite groups  $\cd(\phi\times\phi)=2\cd(\phi)$.
\end{thm}
\end{abstract}

\maketitle

\section{Introduction}

The definition of  cohomological dimension $\cd(G)$ of discrete group $G$ goes back the middle of the  last century. 
The main results of the cohomological dimension theory of groups can be found in the books~\cite{Bie} and \cite{Bro}.

The cohomological dimension $\cd(\phi)$ of a group homomorphism $\phi:G\to H$ was defined a decade ago~\cite{Gr}. It was little studied (\cite{Sc},\cite{DK}), 
since the concept of $\cd(\phi)$ is quite difficult to work with. In this paper we managed to extend some of the classic results of cohomological dimension theory
from groups to group homomorphisms. 

For groups together with the notion of cohomological dimension $\cd(G)$ there are important notions of geometric dimension $\gd(G)$ and
homological dimension $\hd(G)$~\cite{Bro}. The classical result of cohomological dimension theory states that $\cd(G)=\gd(G)$ whenever $\cd(G)\ne 2$.
In the case of $\cd(G)=2$ this equality is the subject of the Eilenberg-Ganea conjecture. The proof of the above equality breaks in two parts: $\cd(G)\ge 3$ (the Eilenberg-Ganea theorem) and $\cd(G)=1$ (the Stallings-Swan theorem)~\cite{Bro}. For group homomorphisms there are counterexamples to both theorems. The Stallings-Swan theorem was disproved by T. Goodwillie \cite{Gr},\cite{DK} with a group homomorphism $\phi:G\to H$ with $\cd(\phi)=1$ and $\gd(\phi)=2$ and the
the Eilenberg-Ganea theorem was disproved in~\cite{DK} with $\cd(\phi)=2$ and $\gd(\phi)=3$.
 In this paper we  show that the Eilenberg-Ganea theorem does not hold for group homomorphisms
with $\cd(G)=n$ for any $n\ge 2$.

Generally, $\hd(G)\le\cd(G)$ and there are groups where the inequality is strict~\cite{Bro}. Studying groups $G$ with $\hd(G)=1$ and $\cd(G)=2$ is  a quite exciting activity (see~\cite{KLL}).
We note that for geometrically finite groups $G$, i.e. groups for which there is a finite classifying CW-complex $BG=K(G,1)$, there is the equality $\hd(G)=\cd(G)$.
In this paper we prove this equality for homomorphisms of geometrically finite groups. 
\begin{thm}
For every homomorphism $\phi:G\to H$ of a geometrically finite group $G$ we have $\hd(\phi)= \cd(\phi)$.
\end{thm}
It is known that the logarithmic law for the cohomological dimension of the product does not hold even for geometrically finite groups~\cite{Dr1}, though
the formula $\cd(G\times G)=2\cd(G)$  holds~\cite{Dr2}.    In this paper we extend the latter equality to group homomorphisms between geometrically finite groups.
\begin{thm}
For every homomorphism $\phi:G\to H$ between geometrically finite groups  $$\cd(\phi\times\phi)=2\cd(\phi).$$
\end{thm}
Among our main tools in proving the above theorems are Davis' trick with aspherical manifolds and a characterization of cohomological dimension of a group homomorphism in terms of projective resolutions which we give in this paper.

\section{Preliminaries}

\subsection{Cohomology with local coefficients}
Let~$\pi=\pi_1(X)$, and let~$\mathbb Z\pi$ be the group ring of~$\pi$. 
We will refer to $\mathbb Z\pi$-module as to $\pi$-module. There is a bijection
between local coefficients systems (locally trivial sheaves) $\A$ on $X$ and $\pi$-modules
where $\A$ corresponds to the stalk $A=\A_x$.
 A map~$f: X \to Y$ and a local coefficient system~$\A$ on~$Y$,
define a local coefficient system on~$X$, denoted~$f^*\A$. Given a pair of coefficient systems~$\A$ and~$\B$ on $X$, the
tensor product~$\A\otimes \B$ is defined by setting~$(\A\otimes
\B)_x=\A_x\otimes\B_x$. On the $\pi$-module level it s the tensor product of abelian groups $A\otimes B$ with the diagonal action of $\pi$.

We recall the definition of the (co)homology groups with local
coefficients via $\pi$-modules \cite{Ha}:

$$H^k(X;\A)\cong H^k(\Hom_{\mathbb Z\pi}(C_*(\wt X), A), \delta)$$

and
$$
H_k(X;\A)\cong H_k(A\otimes _{\mathbb Z\pi}C_*(\wt X), 1\otimes \pa)
$$
where~$(C_*(\wt X), \pa)$ is the chain complex of the universal
cover~$\wt X$ of~$X$ and ~$A$ is the stalk of the local coefficient
system~$\A$, and~$\delta$ is the coboundary operator.  
When $\A$ is a trivial coefficient system the above definition coincides with the usual definition in terms of the chain complex $C_*(X)$ of $X$. Moreover,  there is the following 
intermediate statement.
\begin{prop}\label{cover}
Let $p:X'\to X$ be a normal covering and $\A$ local coefficient system on $X$ such that $p^*\A$ is a trivial system. Then cohomology and homology
groups of $X$ with coefficients in $\A$ can be defined by means of the chain complex $(C_*(X'),\partial)$ as 
$$H^k(X;\A)\cong H^k(\Hom_{\mathbb Z\pi'}(C_*(X'), A), \delta)$$
and
$$
H_k(X;\A)\cong H_k(A\otimes _{\mathbb Z\pi'}C_*(X'), 1\otimes \pa)
$$
where $\pi'=\pi_1(X)/\pi_1(X')$ is the structure group of $p$.
\end{prop}
Proposition~\ref{cover} follows from isomorphisms of (co)chain complexes $$(\Hom_{\mathbb Z\pi}(C_*(\wt X), A), \delta)\cong(\Hom_{\mathbb Z\pi'}(C_*(X'), A), \delta)$$
and $$(A\otimes _{\mathbb Z\pi}C_*(\wt X), 1\otimes \pa)\cong(A\otimes _{\mathbb Z\pi'}C_*(X'), 1\otimes \pa).$$

We refer to \cite{Bre} for the definition of the cup
product
\[
\cup: H^i(X;\A) \otimes H^j(X;\B) \to H^{i+j}(X; \A\otimes \B)
\]
and the cap product
\[
\cap: H_i(X;\A) \otimes H^j(X;\B) \to H_{i-j}(X; \A\otimes \B).
\]

\subsection{The Kunneth Formula and UCF}
For abelian groups $A$ and $B$ we use notation $A\ast B$ for the $Tor(A,B)$.
\begin{thm}[Kunneth~\cite{Bre}] 
Let $\A$ and $\B$ be local coefficient systems on a finite complexes $X$ and $Y$ with $\A_x\ast\B_x=0$. Then there is a natural short exact sequence
$$
0\to\bigoplus_{p}H^p(X;\A)\otimes H^{r-p}(Y;\B)\to H^r(X\times Y;\A\hat\otimes\B)\to \bigoplus_{p}H^{p+1}(X;\A)\ast H^{r-p}(Y;\B)\to 0.
$$
\end{thm}
When $Y=pt$ the Kunneth theorem turns into the Universal Coefficient Formula:
\begin{thm}[UCF~\cite{Bre}]
Let $\A$  be a local coefficient system on a finite complexes $X$ and let $B$ be an abelian group with $A\ast B=0$. Then there is a natural short exact sequence
$$
0\to H^n(X;\A)\otimes B\to H^n(X;\A\otimes B)\to H^{n+1}(X;\A)\ast B\to 0.
$$
\end{thm}

\subsection{Cohomological dimension of groups}
By definition the cohomology of a discrete group $\pi$ with coefficients in a $\pi$-module $A$ is the cohomology of its classifying space,
$H^*(\pi,A)=H^*(B\pi;\A)$ where $\A$ is the corresponding local system. Also we will be using the notation $H^*(B\pi;A)$ for $H^*(\pi,A)$.
It is customary to separate coefficients for group (co)homology by a coma and for the space cohomology by a semicolon. 

The cohomological dimension of a group $\pi$ is defined as follows
$$
\cd(\pi)=\max\{n\mid H^n(\pi,A)\ne 0\ \ \text{for some}\ \ \mathbb Z\pi\text{-module}\  A\}.$$
Similarly the homological dimension is defined as
$$
\hd(\pi)=\max\{n\mid H_n(\pi,A)\ne 0\ \ \text{for some}\ \ \mathbb Z\pi\text{-module}\  A\}.$$
The geometric dimension $\gd(\pi)$ is defined as the minimal dimension of CW-complexes serving as a classifying space $B\pi$ for $\pi$
$$\gd(\pi)=\min\{n\mid \exists\ B\pi  \ \ \text{with}\ \ \dim B\pi=n\}.$$
Consider the short exact sequence $0\to \I_\pi\to\mathbb Z\pi\stackrel{\epsilon}\to\mathbb Z\to 0$ defined by the augmentation map $\epsilon$.
The {\em Berstein-Schwarz class} $\beta_{\pi}\in H^1(\pi,\I_\pi)$ of a discrete group $\pi$ is the image $\delta(1)$ of $$1\in\mathbb Z=H^0(B\pi;\mathbb Z)\stackrel{\delta}\to H^1(B\pi;\I_\pi)$$
under the connecting homomorphism in the coefficients long exact sequence.
Here $\I_\pi$ is the augmentation ideal of the group ring $\Z\pi$~\cite{Be},\cite{Sch}. 
The Berstein-Schwarz class is universal in the following sense.
\begin{thm}[Universality~\cite{DR},\cite{Sch}]\label{universal}
For any cohomology class $\alpha\in H^k(\pi,L)$ there is a homomorphism of $\pi$-modules $\I_\pi^k\to L$ such that the induced homomorphism for cohomology takes $\beta_{\pi}^k\in H^k(\pi;\I_\pi^k)$ to $\alpha$  where $\I_\pi^k=\I_\pi\otimes\dots\otimes \I_\pi$ and $\beta_{\pi}^k=\beta_{\pi}\cup\dots\cup\beta_{\pi}$.
\end{thm}
The Universality Theorem implies 
\begin{cor}\label{cd}
$$\cd(\pi)=\max\{k\mid\beta_\pi^k\ne 0\}.$$
\end{cor}
Everywhere above the background ring $\mathbb Z$ can be replaced by any commutative ring with unit $\k$ and the dimensions $\cd_\k(\pi)$ and $\hd_\k(\pi)$
can be defined accordingly.
Thus, in our notations $\cd(\pi)=\cd_{\mathbb Z}(\pi)$. In this paper we use along with $\mathbb Z$ the fields $\k=\mathbb Z_p$ and $\k=\mathbb Q$.
\begin{prop}\label{integers-group}
For all $\k$ and $\pi$,
$$ \cd_\k(\pi)\le\cd(\pi).$$
\end{prop}
\begin{proof}
This follows from the fact that any \(\k\pi\) module is canonically a \(\Z\pi\) module. So
\begin{equation*}
\begin{aligned}
    \cd_\k \pi &= \max \{n: H^n(\pi; M) \neq 0 \text{ for some \(\k\pi\)-module \(M\)}\} \\
    & \leq \max \{n: H^n(\pi; M) \neq 0 \text{ for some \(\Z\pi\)-module \(M\)}\} \\
    & = \cd \pi.
\end{aligned}
\end{equation*}
\end{proof}
\subsection{Projective resolutions}
Let $\pi$ be a discrete group. A projective resolution $P_*(\pi)$ of $\mathbb Z$ for the group $\pi$ is an exact sequence of projective $\pi$-modules
$$
\dots\to P_n(\pi)\to P_{n-1}(\pi)\to\dots\to P_2(\pi)\to P_1(\pi)\to\mathbb Z\pi\stackrel{\epsilon}\to\mathbb Z.$$
The homology and cohomology of $\pi$ with coefficients in a $\pi$-module $M$ can be defined as homology of the chain and cochain complexes $P_*(\pi)\otimes_\pi M$
and $Hom_\pi(P_*(\pi),M)$.

The cohomological dimension $\cd(\pi)$ equals the shortest length of projective resolution $P_*(G)$~\cite{Bro}.
The same holds true for any commutative ground ring  with unit $\k$.

\section{Cohomological dimension of group homomorphisms}
Let $\phi:G\to H$ be a group homomorphism, and let \(\k\) be a ring. We define (see~\cite{Gr})
$$
\cd(\phi)=\max\{n\mid 0\ne\phi^*:H^n(H,A)\to H^n(G,A)\ \ \text{for some}\ \ \mathbb ZH\text{-module}\  A\}.$$
Similarly the homological dimension is defined as
$$
\hd(\phi)=\max\{n\mid 0\ne\phi^*:H_n(G,A)\to H_n(H,A)\ \ \text{for some}\ \ \mathbb ZH\text{-module}\  A\}.$$

As before, everywhere the background ring \(Z\) can be replaced with a commutative ring \(\k\) with identity, to get \(\cd_\k(\phi)\) and \(\hd_\k(\phi)\). Note that for the identity homomorphism $1_G:G\to G$ we recover the definition of $\cd(G)$ and $\hd(G)$.

Note that the proof of \ref{integers-group} translates directly to the proof of the following:

\begin{prop} \label{integers}
For any commutative ring \(\k\) with identity, and group homomorphism \(\phi : G \to H\) we have
\begin{equation*}
    \cd_\k(\phi) \leq \cd(\phi).
\end{equation*}
\end{prop}

Any group homomorphism $\phi:G\to H$ can be realized by a cellular map $\bar\phi:BG\to BH$ as $\phi=\bar\phi_*:\pi_1(BG)\to\pi_1(BH)$.
The geometric dimension $\gd(\phi)$ of $\phi$ is defined as the minimal $n$ such that $\bar\phi$ can be deformed to the $n$-skeleton $BH^{(n)}$:
$$\gd(\phi)=\min\{n\mid \bar\phi\sim f:BG\to BH^{(n)} \}.$$
The Lusternik-Schnirelmann category $\cat(f)$ of a map $f:X\to Y$ is defined as the minimal number $k$ such that $X$ can be covered by $k+1$ open sets $U_0, U_1,\dots,U_k$
with null-homotopic restrictions $f|_{U_i}:U_i\to Y$ for all $i$. This definition can be extended to group homomorphisms $\phi:G\to H$ as $\cat(\bar\phi)$~\cite{Sc}.
The following is well-known (see for example Proposition 4.3 in~\cite{Dr3}).
\begin{prop}
For a group homomorphism $\phi:G\to H$ with its realizing map the $\bar\phi:BG\to BH$ the following are equivalent:

(1) $\gd(\phi)\le n$;

(2) $\cat(\bar\phi)\le n$. 
\end{prop}
We note that in the case of the identity homomorphism $1_G:G\to G$ our definition of the geometric dimension gives the category $\cat(G)$ which by the Eilenberg-Ganea theorem~\cite{EG}
coincides with $\cd(G)$. The later equals $\gd(G)$ with a potential exception when $\cd(G)=2$.

We recall that any group homomorphism $\varphi : G\to H$ defines a chain map between projective resolutions $\varphi_*:P_*(G)\to P_*(H)$ which induces homomorphism of (co)homology groups~\cite{Bro}.
A chain homotopy between two chain maps $\varphi_*,\psi_*:P_*(G)\to P_*(H)$ is a sequence of homomorphisms $D_*:P_*(G)\to P_{*+1}(H)$ such that
$\partial D_*+D_*\partial=\varphi_*-\psi_*$. Two chain homotopic homomorphisms induce the same homomorphisms for homology and cohomology~\cite{Ha}.
\begin{thm}\label{chain}
Let $\varphi:G\to H$ be a group homomorphism with $\cd(\varphi)=n$ and let $\varphi_*:(P_*(G),\partial_*)\to (P_*(H),\partial_*')$  be the chain map between projective resolutions of $\mathbb Z$
for $G$ and $H$ induced by $\varphi$.
Then $\varphi_*$ is chain homotopic to $\psi_*:P_*(G)\to P_*(H)$ with $\psi_k=0$ for $k>n$.
\end{thm}
\begin{proof}

Let $d'_{n+1}:P_{n+1}(H)\to K_{n}=\Im(\partial_{n+1}')$ be the range restriction for $\partial_{n+1}'$. We have the following  commutative diagram:

\[ \begin{tikzcd}
    & \cdots \arrow{r} 
    & P_{n+2}(G) \arrow{r}{\partial_{n+2}} \arrow{d} {\varphi_{n+2}}
    & P_{n+1}(G) \arrow{r}{\partial_{n+1}} \arrow{d} {\varphi_{n+1}}
    & P_{n}(G)   \arrow{r}                 \arrow{d} {\varphi_{n}}
    & \cdots \\
    & \cdots \arrow{r} 
    & P_{n+2}(H) \arrow {r}{\partial_{n+2}'} \arrow{dr} {d'_{n+2}}
    & P_{n+1}(H) \arrow {r}{\partial_{n+1}'} \arrow{dr} {d'_{n+1}}
    & P_{n}(H)   \arrow {r}
    & \cdots \\
    & \cdots
    & K_{n+2} \arrow {u}{\subset}
    & K_{n+1}\arrow {u}{\subset}
    & K_{n}\arrow {u}{\subset}
    & \cdots
\end{tikzcd} \]
Since \(\partial'_{n+2} d'_{n+1} = 0\), we obtain
    that \(d'_{n+1}\) is a cocycle in the
    cochain complex \(\Hom_H(P_*(H), K_n)\).
    Therefore, it represents an element \([d'_{n+1}]\in H^{n+1}(H,
    K_n)\). Since \(\cd(\varphi) = n\), we get that \(\varphi^*[d'_{n+1}] = 0\).
    That is, $$d_n := d'_{n+1}\varphi_{n+1} : P_{n+1}(G) \to K_n$$ is a coboundary
    in the cochain complex \(\Hom_G(P_*(G), K_n)\). So there exists a map \(h_n: P_n(G)
    \to K_n\) such that \(h_n\partial_{n+1} = d_n\). But since \(P_n(G)\) is a projective \(G\)-module and \(d'_{n+1}\) is  surjective, 
the map \(h_n\) can be lifted to a $G$-homomorphism \(D_n : P_n(G) \to P_{n+1}(H)\),
    such that \(d'_{n+1}D_n = h_n\). Note that the \(D_n\) as constructed has
    the property 
    $$
        \partial'_{n+1}(\varphi_{n+1} - D_n\partial_{n+1})
          = \partial'_{n+1}\varphi_{n+1} -
         \partial'_{n+1}D_n\partial_{n+1} \\
         = \partial'_{n+1}\varphi_{n+1} - h_n\partial_{n+1} = 0.$$

    Now assume that we have constructed \(D_{n+i-1}\) such that
    $$
        \partial'_{n+i}(\varphi_{n+i} - D_{n+i-1}\partial_{n+i}) = 0.$$
    Since the bottom row is exact, $$\Im( \varphi_{n+i} - D_{n+i-1}\partial_{n+i}) \subset K_{n+i}=\Im(\partial_{n+i+1}').$$ Since $P_{n+i}(G)$ is a projective $G$-module  the $G$-homomorphism $$\varphi_{n+i} -
    D_{n+i-1}\partial_{n+i}: P_{n+i}(G)\to K_{n+i}$$ can be lifted to a $G$-homomorphism \(D_{n+i} : P_{n+i}(G) \to
    P_{n+i+1}(H)\). Since
    \begin{equation}\label{6}
        \partial'_{n+i+1} D_{n+i} = \varphi_{n+i} - D_{n+i-1}\partial_{n+i}
    \end{equation}
    we obtain
    \begin{equation*} \begin{aligned}
        \partial'_{n+i+1}(\varphi_{n+i+1} - D_{n+i}\partial_{n+i+1})
        & = \partial'_{n+i+1}\varphi_{n+i+1} -
        \partial'_{n+i+1}D_{n+i}\partial_{n+i+1} \\
        & = \partial'_{n+i+1}\varphi_{n+i+1} -
        (\varphi_{n+i} - D_{n+i-1}\partial_{n+i})\partial_{n+i+1} \\ 
        & = \partial'_{n+i+1}\varphi_{n+i+1} -
        \varphi_{n+i}\partial_{n+i+1} -
        D_{n+i-1}\partial_{n+i}\partial_{n+i+1} \\
        & = \partial'_{n+i+1}\varphi_{n+i+1} -
        \varphi_{n+i}\partial_{n+i+1} \\
        & = 0 
    \end{aligned} \end{equation*}
    Here we are using the facts that \(\partial\partial = 0\) and \(\varphi_*\)
    is a chain map.

    Therefore we can inductively construct \(D_{k}\)'s for all \(k\ge n\). Define
    \(D_k: P_k(G) \to P_{k+1}(H)\) to be the zero homomorphisms for all \(k<n\). Therefore we
    have constructed a chain homotopy \(D_* : P_*(G) \to P_{*+1}(H)\) between \(\varphi_*\)
    and a chain map \(\psi_*\) defined by
    $$\psi_k = \varphi_k - \partial'_{k+1} D_k - D_{k-1} \partial_k.$$
    In view of~(\ref{6}) we obtain that \(\psi_k = 0\) for all \(k>n\).
\end{proof}

\section{Reduction to a map of aspherical manifold}

\begin{prop}\label{comp}
Let $F$ be any of the numerical invariants $\cd_\k$, $\hd_\k$, or $\gd$. Then for homomorphisms $A\stackrel{g}\to G\stackrel{f}\to H$ $$F(f\circ g)\le\min\{F(f),F(g)\}.$$
\end{prop}
\begin{proof}
The statement is obvious for the cohomological and homological dimensions. Clearly, $\gd(f\circ g)\le \gd(f)$. We may assume that the map $\bar f:BH\to BA$ realizing $f$
is cellular. Then the inequality $\gd(f\circ g)\le \gd(g)$ follows.
\end{proof}
\begin{cor}
Suppose that $r:BA\to BG$ is a retraction. Then for any $f:BG\to BH$ and any above choice of the numerical invariant $F$, $F(f\circ r)=F(f)$.
\end{cor}
\begin{proof}
Let $i:BG\to BA$ be the inclusion for which $r\circ i=1$. Then $$F(f\circ r)\le F(f)=F((f\circ r)\circ i)\le F(f\circ r).$$
\end{proof}
\begin{cor}\label{mnfld}
For every homomorphism $\phi:G\to H$ of  a geometrically finite group $G$ there is a  closed aspherical orientable manifolds $M$ containing $BG$ as a retract and a map $g:M\to BH$ such that
$\phi=(g|_{BG})_*:\pi_1(BG)\to\pi_1(BH)$ and $F(\phi)=F(g)$ for the above choice of $F$.
\end{cor}
\begin{proof}
By Davis' trick~\cite{D} for every group $G$ with finite complex $BG$ there is a closed aspherical orientable manifold containing $BG$ as a retract.
\end{proof}

\begin{thm}\label{epi}
Let homomorphism $\phi:G\to H$ factor as $\phi=j\circ\phi'$ where $\phi'$ is surjective and $j$ is injective. Then for any choice $F$ of the numerical invariants $\cd_\k$, $\hd_\k$, or $\gd$
we have $F(\phi)=F(\phi')$.
\end{thm}
\begin{proof}
In view of the equality $\gd(\phi)=\cat(\phi)$, Theorem 3.2 and Theorem 3.3 in~\cite{DK}, the statement is proven when $F=\gd$ ands $F=\cd$.

By Proposition~\ref{comp} $\hd(\phi)\le \hd(\phi')$. Let $\hd(\phi')=k$. We show that $\hd(\phi)\ge k$.
Let $\pi=\phi'(G)$ and let $\phi'_*:H_k(G,M)\to H_k(\pi,M)$ be a nonzero homomorphism for some $\pi$-module $M$. 
Let $i$ denote the inclusion $M\cong 1\otimes M\stackrel{\subset}\to\mathbb ZH\otimes_\pi M=\Ind^H_\pi M$  of $M$ into the induced $H$-module.
We consider the following commutative diagram generated by $i$, the inclusion $j:\pi\to H$, and $\phi'$
$$
\begin{CD}
H_k(G,M) @>i_*>> H_k(G,\Ind_\pi^HM) @.\\
@V\phi'_*VV @V\phi'_*VV @.\\
H_k(\pi,M) @>i_*>> H_k(\pi,\Ind_\pi^HM) @>j_*>> H_k(H,\Ind^H_\pi M).\\
\end{CD}
$$
The bottom composition $j_*i_*$  is the Shapiro Lemma isomorphism~\cite{Bro}. Therefore, $j_*i_*\phi'_*\ne 0$. Thus, the composition
$$\phi_*=j_*\phi'_* :H_k(G,\Ind_\pi^HM)\to H_k(H,\Ind^H_\pi M)$$ is not zero. Hence  $\hd(\phi)\ge k$.

The same proof works for any ground ring $\k$.
\end{proof}

\begin{cor}\label{any}
Let homomorphism $\phi:G\to H$ factor as $\phi=j\circ\phi'$ where  $j$ is injective. Then for any choice $F$ of the numerical invariants $\cd_\k$, $\hd_\k$, or $\gd$
we have $F(\phi)=F(\phi')$.
\end{cor}
\begin{proof}
Let $\phi':G\to H'$, $j:H'\to H$, and let $\phi'':G\to\Im\phi$ denote the range restriction of $\phi$. Let $i:\Im(G)\to H'$ be the inclusion. Then by Theorem~\ref{epi} 
applied to $\phi=(ij)\phi''$ and $\phi'=i\phi''$, 
$$F(\phi)=F((ij)\phi'')
=F(\phi'')=F(i\phi'')=F(\phi').$$
\end{proof}

\begin{cor}\label{bothmnflds}
If the group $H$ in Corollary~\ref{mnfld} is geometrically finite, then there is a closed orientable aspherical manifold $N$ containing $BH$ and a map  $f:M\to N$ such that 
$f(BG)\subset BH$, $\phi=(f|_{BG})_*:\pi_1(BG)\to\pi_1(BH)$ and $F(\phi)=F(g)$ for the above choice of $F$.

\end{cor}
\begin{proof}
If $H$ is geometrically finite we consider the embedding $i_H:BH\to N$ of $BH$ into a closed aspherical orientable manifold from the Davis' trick.
Let $f=i_H\circ g:M\to N$. By Theorem~\ref{epi}, $F(f)=F(g)$.
\end{proof}

We call a group homomorphism $\phi:G\to H$ a {\em subhomomorphism} of a group homomorphism $\phi':G'\to H'$ if there are inclusions as a subgroup $i:G\to G'$ and $j:H\to H'$
such that $\phi=j\phi' i$.
\begin{prop}\label{sub}
Let $\phi:G\to H$ be a subhomomorphism of $\phi':G'\to H'$. Then $F(\phi)\le F(\phi')$ for the above choice of $F$.
\end{prop}
\begin{proof}
By Corollary~\ref{any}, the equality $j\phi=\phi' i$, and  Proposition~\ref{comp} we obtain $$F(\phi)=F(j\phi)=F(\phi' i)\le F(\phi').$$
\end{proof}

\begin{prop}\label{detect}
Suppose that a local coefficient system $\A$ on a CW complex $X$ pulls back to a trivial system  $(p')^*\A=X'\times A$  on $X'$ by a normal covering map $p':X'\to X$. Then
every homology class $a\in H_k(X;\A)$ can be detected by a cohomology class $\gamma\in H^k(X;\B)$ with some local coefficient system $\B$  with $(p')^*\B$ trivial on $X'$, that is
$0\ne a\cap\gamma\in H_0(X,\A\otimes\B)$.
\end{prop}
\begin{proof}
Consider the
chain complex
\[
\CD \dots @>>> C_{k+1}(X') @>\partial_{k+1}'>>C_k(X') @>\pa_k'>>
C_{k-1}(X') @>>> \dots\ \ .
\endCD
\]
We set~$B': =C_k(X')/\Im \partial_{k+1}'$.  Let~$\B$ be
the corresponding local system on~$X$.  Let $p:\wt X\to X$ be the universal covering and let $p=p'\circ q$.
The map $q$ defines the following commutative diagram:
\[
\CD 
C_{k+1}(\wt X) @>\partial_{k+1}>>C_k(\wt X)@>\wt f>>\wt B @>>>0\\
@Vq_*VV @Vq_*VV @Vq_*VV @.\\
C_{k+1}(X') @>\partial_{k+1}'>>C_k(X')@>f'>>B'@>>> 0  
\endCD
\]
where $\wt B: =C_k(\wt X)/\Im \partial_{k+1}$. Since the chain complex $C_*(\wt X)$ is exact, $\wt B=C_k(\wt X)/\Ker \partial_{k}=\Im\partial_k$.
Then $\wt f$ can be identified as  the range restriction of $\partial_k$.
We define $f=f'\circ q_*:C_k(\wt X)\to B'$. Since $$\delta f(x)=f(\partial_{k+1}(x))=q_*((\wt f\circ\partial_{k+1})(x))=q_*(\partial_k\circ\partial_{k+1})(x))=0,$$
the homomorphism~$f$ can be regarded as a~$k$-cocycle with
values in~$B'$. Let~$\gamma:=[f]\in
H^k(X;\B)$ be the cohomology class of~$f$.  

Now we prove that
$
a\cap\gamma \not=0.
$
In view of Proposition~\ref{cover} we can represent the class~$a$ by a cycle
$
z\in A\otimes  C_k(\wt X)$ where $A$ be the stalk of $\A$.
Since~$z\notin \Im(1\otimes \pa_{k+1})$, we conclude that
$
(1\otimes \wt f)(z)\ne 0\in A\otimes \wt B$.

The tensor product of the above diagram with $A$ over the group ring $\Z\pi$ where $\pi=\pi_1(X)$
gives the following commutative diagram with exact rows
\[
\CD 
A\otimes_\pi C_{k+1}(\wt X) @>1\otimes\partial_{k+1}>>A\otimes_\pi C_k(\wt X)@>1\otimes\wt f>>A\otimes_\pi\wt B @>>>0\\
@V1\otimes q_*V\cong V @V1\otimes q_*V\cong V @V1\otimes q_*V\cong V @.\\
A\otimes_\pi C_{k+1}(X') @>1\otimes\partial_{k+1}'>>A\otimes_\pi C_k(X')@>1\otimes f'>>A\otimes_\pi B'@>>> 0  
\endCD
\]
Since the action of the group $\pi'=\pi_1(X')\subset\pi$ on $A$ is trivial, the left two vertical arrows in the diagram are isomorphisms. 
By the Five Lemma the right vertical arrow is an isomorphism as well.
Then 
$$
(1\otimes f)(z)=(1\otimes f')(1\otimes q_*)(z)=(1\otimes q_*)(1\otimes\wt f)(z)\ne 0\in  A\otimes_\pi B'=H_0(X;\A\otimes\B).
$$

Thus, for the cohomology class~$\gamma$ of~$f$ we have~$a\cap \gamma \ne 0$.
\end{proof}

\begin{thm}\label{hd}
For any choice of commutative ring with unit $\k$, for every homomorphism $\phi:G\to H$ of geometrically finite group $G$, $$\hd_\k(\phi)= \cd_\k(\phi).$$
\end{thm}
\begin{proof}
Since the proof in the case of general ground ring $\k$ is the same as for $\mathbb Z$, one may assume that everything below is performed over the ring of integers.
In view of Theorem~\ref{epi} it suffices to prove this theorem for the case when $\phi$ is an epimorphism.
Let $r:M\to BG$ be a retraction of  a closed orientable aspherical  manifold as in Corollary~\ref{mnfld}.
We define $g=\bar\phi r:M\to BH$ where $\bar\phi:BG\to BH$ is a map realizing $\phi$ on the fundamental groups.
In view of Corollary~\ref{mnfld}, it suffices to show that $\hd(g)=\cd(g)$. Let $M'=g^*EH$ be the pull-back of the universal covering of $BH$.
Suppose that $\cd(g)=k$. Then $g^*(\beta_H^k)\ne 0$. By the Poincare Duality with local coefficients, $a=[M]\cap g^*(\beta_H^k)\ne 0$.
We note that the local coefficient system $\A$ on $M$ for the class $g^*(\beta_H^k)$ pulls back by $g$ from the system on $BH$ defined by the $H$-module $\I_H^k$.
The same system serves as coefficients for the dual homology class $a\in H_{|a|}(M;\A)$. Since $EH$ is contractible, the system $\A$ pulls back to a
trivial system on $M'$ by the covering $M'\to M$.
By Proposition~\ref{detect} there is a cohomology class $\gamma\in H_{|a|}(M;\B)$ with a local system $\B$ on $M$, that pulls back to a trivial system on $M'$, detecting the homology class $a$. 
Thus, $$([M]\cap\gamma)\cap g^*\beta_H^k\in H_0(M;\A\otimes\B)\ne 0.$$ Since all coefficients are coming from $BH$, 
the induced homomorphism $g_*$ for homology is well-defined. Since $g$ is an epimorphism of the fundamental groups, $g_*$ is an isomorphism in dimension 0.
Thus, we obtain
$$0\ne g_*(([M]\cap\gamma)\cap g^*\beta_H^k)=g_*([M]\cap\gamma)\cap\beta_H^k.$$
Hence,  $g_*([M]\cap\gamma)\ne 0$. Note that $[M]\cap\gamma\in H_k(M;\B)$. Thus, $\hd(g)\ge k$.

By Theorem~\ref{chain} the chain map $g_*: C_*(\wt M)\to C_*(EH)$ is chain homotopic to a chain map $q_*$ with $q_i=0$ for $i>k$. Then for any $H$-module
$1_M\otimes q_i=0$ for $i>k$. Therefore, the induced homomorphism of homology is 0 in dimension greater than $k$. Thus, $\hd(g)\le k$.
\end{proof}
We note that the proof of the inequality $\hd(\phi)\le\cd(\phi)$ works for any group homomorphisms.

\section{Cohomological dimension with respect to a field}

For a discrete group $\pi$ we consider the chain complex $(C_*(E\pi),\partial_*)$. We use notations $C_k=C_k(\pi)=C_k(E\pi)$ for the group of $k$-chains and 
$B_k=B_k(\pi)=\Im\partial_k$
for the group of $(k-1)$-boundaries. Since the chain complex $(C_*(E\pi),\partial_*)$ is acyclic, there are short exact sequences
\begin{equation}\label{boundary}
0\to B_{k+1}\to C_k\to B_k\to 0.
\end{equation}
\begin{prop}\label{Z}
Let $\Lambda=\pi_1(M)$ be the fundamental group of a closed orientable aspherical $n$-manifold $M$. Then  $H^k(\Lambda,B_k(\Lambda))=\Z$ for $k\le n$.
\end{prop}
\begin{proof}
We may assume that $M$ has one $n$-dimensional cell and, hence, $C_n=\Z\Lambda$.
Note that $B_n\cong C_n$ and hence $H^n(\Lambda,B_n)=H^n(M;\Z\Lambda)=H^n_c(\wt M;\Z)=\Z$.

By results of McMillan~\cite{Mc}, McMillan-Zeeman~\cite{McZ}, and Stallings~\cite{St} proven almost simultaneously in early 60s, we obtain $\wt M\times\R^m\cong\R^{n+m}$ for 
some (all) $m\ge 1$. Hence 
the reduced suspension $\Sigma^m(\alpha\wt M)$ of the one point compactification $\alpha\wt M$ of the universal covering $\wt M$ of $M$ is homeomorphic to $S^{n+m}$.
Therefore, $$H^k(\Lambda,\Z\Lambda)=H^k_c(\wt M;\Z)=\bar H^k(\alpha\wt M)=0$$ for $k<n$. Since  the chai groups $C_i(\pi)=\oplus^{m_i}\Z\Lambda$ are finitely generated free $\Lambda$-modules, we obtain from the coefficient exact sequence  (\ref{boundary}), $H^{k-1}(\Lambda, B_{k-1})= H^{k}(\Lambda, B_{k})$ for $k<n$. Then the equality $H^0(\Lambda,B_0)=\Z$ implies
the equality $H^k(\Lambda,B_k)=\Z$ for $k<n$.
\end{proof}

The following Proposition is a refinement of the inequality $\hd(\pi)\le \cd(\phi)$.
\begin{prop}\label{refine}
Let $\phi:\Gamma\to \pi$ be a group homomorphism. Assume that $\phi_*:H_k(\Gamma,A)\to H_k(\pi,A)$ is a nontrivial homomorphism for some $\pi$-module
$A$. Then $\phi^*:H^k(B\pi;B_k(\pi))\to H^k(B\Gamma;B_k(\pi))$ is a nontrivial homomorphism.
\end{prop}
\begin{proof}
Clearly, the natural epimorphism $f:C_k(\pi)\to B_k(\pi)=B_k$  is a cocycle. Let $\alpha=[f]\in H^k(B\pi;B_k)$. We show that $\phi^*(\alpha)\ne 0$.
In fact we show that $a\cap\phi^*(\alpha)\ne 0$ for any $a\in H_k(\Gamma,A)$ with $\phi_*(a)\ne 0$.
Let $z\in A\otimes_\pi C_k(\Gamma)$ be a cycle representing $a$. Then $z'=(1\otimes\phi_*)(z)\in A\otimes_\pi C_k(\pi)$ is a cycle representing $\phi_*(a)\ne 0$.
Therefore, $z'\notin\Im(1\otimes\partial_{k+1})$. Hence, $$0\ne (1\otimes f)(z')\in A\otimes_\pi B_k=H_0(\pi,A\otimes B_k).$$ Thus, $\phi_*(a)\cap\alpha\ne 0$.
 Note that $\phi_*(a)\cap\alpha=\phi_*(a\cap\phi^*(\alpha))$. Therefore, $a\cap\phi^*(\alpha)\ne 0$ and hence $\phi^*(\alpha)\ne 0$.
\end{proof}

\begin{thm} \label{k-exist}
For any homomorphism $\varphi:  \Gamma\to\pi$ between geometrically
finite groups, there is a field $\k$ such that
$\cd(\varphi) = \cd_\k(\varphi)$.
\end{thm}

\begin{proof}
By Corollary~\ref{bothmnflds} we may assume that $\varphi$ is realized by a map of closed aspherical orientable manifolds. 
Let $\cd(\varphi) = k$. Then by Theorem~\ref{hd}, $\hd(\varphi)=k$. By Proposition~\ref{refine}
$$\varphi^*:H^k(B\pi;B_k(\pi))\to H^k(B\Gamma;B_k(\pi))$$ is a nontrivial homomorphism.
By virtue of  Proposition~\ref{Z} the image $$A = \varphi^*(H^k(\pi, B_k(\pi)))\ne 0$$ is a cyclic group.

Case 1: The group $A=\Z$. Then $A \otimes \Q \neq 0$, and
considering the $\Q\Gamma$-module $B_k(\pi) \otimes \Q$, we see that
$\cd_\Q(\varphi) \geq k$. Using Proposition~\ref{integers} we get $\cd_\Q(\varphi) =
k$, and we are done.

Case 2: The group  $A$ is of finite order $m$. Let prime $p$ be a divisor of $m$.  Then the homomorphism 
$$
H^k(\pi, B_k(\pi))\otimes\Z_p\to A\otimes\Z_p
$$
is a nontrivial epimorphism. Since the group $B_k(\pi)\subset C_{k-1}(\pi)$ is torsion free, we have $B_k(\pi)\ast\Z_p=0$. Therefore, the Universal Coefficient Formula (UCF) holds
for $B\pi$. By the UCF there is a commutative diagram of exact sequences
$$
\begin{CD}
0 @>>> H^k(B\pi;B_k(\pi)\otimes\Z_p @>>> H^n(B\pi;B_k(\pi)\otimes\Z_p)\\
@. @V\ne 0VV @V\varphi_p^*VV\\
0 @>>> H^k(B\Gamma;B_k(\pi)\otimes\Z_p @>>> H^n(B\Gamma;B_k(\pi)\otimes\Z_p)\\
\end{CD}
$$
which implies that the homomorphism $$\varphi_p^*:H^n(B\pi;B_k(\pi)\otimes\Z_p)\to H^n(B\Gamma;B_k(\pi)\otimes\Z_p)$$ is nonzero.

Therefore, taking $M = B_k(\pi) \otimes \Z_p$ as
our $\Z_p\pi$-Module, we see that the map $H^k(\pi, M) \to H^k(\Gamma,
M)$ induced by $\varphi$ is nonzero. Therefore, $\cd_{\Z_p} (\varphi) \geq k$.
Again, by Proposition~\ref{integers} we conclude that $\cd_{\Z_p} (\varphi)
=\cd(\varphi)$.
\end{proof}

\section{Cohomological dimension of the product}

\begin{lem}\label{lineq}
 For any two homomorphisms $\phi:  \Gamma\to\pi$ and $\phi':\Gamma'\to\pi'$ and any commutative ring $R$ there is the inequality $\cd_R(\phi\times\phi')\le\cd_R\phi+\cd_R\phi'$.
\end{lem}
\begin{proof}
Consider projective resolutions $P_*(\Gamma)$, $P_*(\Gamma')$, $P_*(\pi)$, and $P_*(\pi')$ of $R$ for $R\Gamma$, $R\Gamma'$, $R\pi$, and $R\pi'$ modules respectfully.
Let $\phi_*:P_*(\Gamma)\to P_*(\pi)$  and $\phi_*':P_*(\Gamma')\to P_*(\pi')$  be  chain maps between projective resolutions generated by $\phi$ and $\phi'$.
Let $\cd_R(\phi)=m$ and $\cd_R(\phi')=n$. By Theorem~\ref{chain} $\phi_*$ is chain homotopic to a chain map $\psi_*$ with $\psi_i(P_i(\Gamma))=0$ for $i>m$.
Similarly, $\phi_*'$ is chain homotopic to a chain map $\psi_*'$ with $\psi_j(P_j(\Gamma'))=0$ for $j>n$. Then the chain map 
$$(\phi\times\phi')_*:(P(\Gamma)\otimes_R P(\Gamma'))_*\to (P(\pi)\otimes_R P(\pi'))_*$$ is chain homotopic to
$(\psi\times\psi')_*$. Note that $(\psi\times\psi')_k=0$ for $k>m+n$. Therefore, the induced homomorphism $$(\psi\times\psi')^*=(\phi\times\phi')^*:H^k(\pi\times\pi',M)\to H^k(\Gamma\times\Gamma', M)$$ is 0 for $k>m+n$ for any $R(\pi\times\pi')$-module $M$. Hence $\cd_R(\phi\times\phi')\le m+n$.
\end{proof}

\begin{lem} \label{k-ineq}
 For any two homomorphisms $\varphi:  \Gamma\to\pi$ and  $\psi:  \Gamma'\to\pi'$
between geometrically finite groups  and for any field $\k$ there is the inequality $$\cd_\k(\varphi
\times \psi) \geq \cd_\k(\varphi)+\cd_\k(\psi).$$
\end{lem}

\begin{proof}
We recall that there is
the Kunneth Formula for the sheaf cohomology of compact spaces~\cite{Bre} which in the case of a field degenerates into a natural isomorphism. Thus, for any geometrically finite group $G$ and any $\k
G$-module $M$ for a field $\k$ there is an isomorphism:
\[
   \bigoplus_{p + q = r} H^p(G; M) \otimes_\k H^q (G; M) 
  \to H^{r} (G \times G; M \otimes_\k M') 
.\]
 Let $m = \cd_\k(\varphi)$ and  Let $n = \cd_\k(\psi)$.
Let $M$ to be a
$\k\pi$-module for which $\varphi^*:H^n(\pi;M) \to H^n(\Gamma;M)$ is nonzero
and let $M'$ to be a
$\k\pi'$-module for which $\psi^*:H^n(\pi';M') \to H^n(\Gamma';M')$ is nonzero.
We obtain the following commutative diagram:

\[ \begin{tikzcd}
  0 \arrow{r} 
    & \bigoplus_{p+q=r} H^p(\pi, M) \otimes_\k H^q (\pi', M') \arrow{r}
      \arrow{d} {\bigoplus \varphi^* \otimes \psi^*}
    & H^{r} (\pi \times \pi',M \otimes_\k M') \arrow{r} 
      \arrow{d} {(\varphi \times \psi)^*}
    & 0 \\
  0 \arrow{r} 
    & \bigoplus_{p+q=r} H^p(\Gamma; M) \otimes_\k H^q (\Gamma'; M) \arrow{r} 
    & H^{r} (\Gamma' \times \Gamma; M \otimes_\k M') \arrow{r} 
    & 0 \\
\end{tikzcd} \]
For $r=m+n$ on the left hand side the summand with $p=m$ and $q=n$ gives us a nonzero
homomorphism of vector spaces over $\k$. By
commutativity, we get that the right vertical map is also nonzero in dimension $m+n$. Therefore,
$\cd_\k(\varphi \times \psi) \geq m+n$.
\end{proof}
\begin{cor}
For any two homomorphisms $\varphi:  \Gamma\to\pi$ and  $\psi:  \Gamma'\to\pi'$
between geometrically finite groups  and for any field $\k$ there is the equality $$\cd_\k(\varphi
\times \psi) = \cd_\k(\varphi)+\cd_\k(\psi).$$
\end{cor}
\begin{proof}
Apply Lemma~\ref{lineq} and Lemma~\ref{k-ineq}.
\end{proof}
\begin{cor}\label{gineq}
 For any homomorphism $\varphi:  \Gamma\to\pi$ between two  geometrically finite groups there is the inequality $\cd(\varphi
\times \varphi) \geq 2\cd(\varphi)$.
\end{cor}
\begin{proof}
By Theorem~\ref{k-exist} there is a field $\k$ such that $\cd(\varphi)=\cd_\k(\varphi)$. Then in view of Proposition~\ref{integers} $$\cd(\varphi\times\varphi)\ge\cd_\k(\varphi\times\varphi)=  2\cd_\k(\varphi)= 2\cd(\varphi).$$
\end{proof}

\begin{thm}
 For any homomorphism $\varphi:  \Gamma\to\pi$ between two  geometrically finite groups there is the equality $\cd(\varphi
\times \varphi)= 2\cd(\varphi)$.
\end{thm}
\begin{proof}
Apply Lemma~\ref{lineq} with $R=\mathbb Z$ and Corollary~\ref{gineq}.
\end{proof}

\section{On the analog of the Eilenberg-Ganea theorem for homomorphisms}

 The Eilenberg-Ganea equality~\cite{Bro} $\cd(\pi)=\gd(\pi)$ holds true whenever $\cd(\pi)\ge 3$. The Eilenberg-Ganea conjecture
extends this equality to the case of $\cd(\pi)=2$. A potential counter-example should have $\gd(\pi)=3$.
In~\cite{DK} a map $f:W^4\to T^3$ of an aspherical 4-manifold $W^4$ onto a 3-torus is constructed satisfying $\cd(f)=2$ and $\gd(f)=3$.
Below we show the fact that the numbers 2 and 3  in this example are the same as in the Eilenberg-Ganea conjecture is rather coincidental.

We recall the notation $$\pi^k_s(X)=\lim_{\rightarrow}[\Sigma^nX,S^{n+k}]$$ for the stable $k$-cohomotopy group of $X$.
\begin{prop}\label{iterate}
Suppose that the map $f:W\to T$ induces a nontrivial homomorphism $f^*:\pi_s^k(T)\to \pi_s^k(W)$ and satisfies $\cd(f)<k$.
Then the map $f\times 1:W\times S^1\to T\times S^1$ induces a nontrivial homomorphism $$(f\times 1)^*:\pi_s^{k+1}(T\times S^1)\to \pi_s^k(W\times S^1)$$ and satisfies the inequality $\cd(f\times 1)<k+1$.
\end{prop}
\begin{proof}
Let $S^1=I_+\cup I_-$ with $I_+\cap I_-=S^0=\{x_0,-x_0\}$. We use the relative
Mayer-Vietoris sequence for $\pi^*_s$ to show the isomorphism $$\pi^{i-1}_s(Y)=\pi^{i-1}_s(Y\times -x_0)=\pi^{i-1}_s(Y\times S^0, Y\times x_0)\to\pi^i_s(Y\times S^1,Y\times x_0).$$ The retraction $Y\times S^1\to Y\times x_0$
defines the natural on $Y$ splitting $$\pi^i_s(Y\times S^1)=\pi^i_s(Y)\oplus\pi^{i-1}_s(Y\times S^1, Y\times x_0)=\pi^i_s(Y)\oplus\pi^{i-1}_s(Y).$$
The commutative diagram 
$$
\begin{CD}
\pi^i_s(T\times S^1) @>\cong>> \pi^i_s(T)\oplus\pi^{i-1}_s(T)\\
@V(f\times 1)^*VV @Vf^*\oplus f^*VV\\
\pi^i_s(W\times S^1) @>\cong>> \pi^i_s(W)\oplus\pi^{i-1}_s(W)\\
\end{CD}
$$
with $i=k+1$ implies that the homomorphism $$(f\times 1)^*:\pi_s^{k+1}(T\times S^1)\to \pi_s^k(W\times S^1)$$ is not 0.

The inequality $\cd(f\times 1)<k+1$ follows from Lemma~\ref{lineq}.
\end{proof}

\begin{cor}
For any $k>2$ there is a map between aspherical manifolds $f_k:W^{k+1}\to T^k$ with $\cd(f)<k$ and $\gd(f)=k$.
\end{cor}
\begin{proof}
We note that the map $f:W^4\to T^3$ from~\cite{DK} induces a nontrivial isomorphism $f^*:\pi_s^3(T)\to \pi_s^3(W^4)$. We  cross it with $S^1$ and apply Proposition~\ref{iterate} . Then we  iterate this process  until the range is a $k$-torus $T^k$.
\end{proof}


\begin{thebibliography}{[Bre]}


\bibitem[Be]{Be}
I. Berstein,  On the Lusternik-Schnirelmann category of Grassmannians. Math. Proc.
Camb. Philos. Soc. 79  (1976) 129-134.

\bibitem[Bie]{Bie} R. Bieri, Homological dimension of discrete groups, Quinn Mary College, 1976.


\bibitem[Bre]{Bre}
G. Bredon, Sheaf Theory. \emph{Graduate Text in Mathematics}, \textbf{170},
Springer, New York Heidelberg Berlin, 1997.

\bibitem[Bro]{Bro} K. Brown, Cohomology of Groups. \emph{Graduate Texts in Mathematics},
\textbf{87} Springer, New York Heidelberg Berlin, 1994.

\bibitem[D]{D} M. W. Davis, Groups generated by reflections and aspherical manifolds not covered by Euclidean space. Ann. of Math. (2) 117 (1983), no. 2, 293-324.

\bibitem[Dr1]{Dr1} A. Dranishnikov,  On the virtual cohomological dimensions of Coxeter groups. Proc. Amer. Math. Soc. 125 (1997), no. 7, 1885-1891.

\bibitem[Dr2]{Dr2} A. Dranishnikov, On dimension of product of groups. Algebra Discrete Math. 28 (2019), no. 2, 203-212.

\bibitem[Dr3]{Dr3} A. Dranishnikov, The LS category of the product of lens spaces, Algebr. Geom. Topol. 15 (2015) no 5, 2985-3010.

\bibitem[DK]{DK} A. Dranishnikov, N. Kuanyshov, On the LS-category of group homomorphisms, Preprint 2022, arXiv:2203.03734.

\bibitem[DR]{DR} A. Dranishnikov, Yu. Rudyak, On the Berstein-Svarc theorem in dimension 2. Math. Proc. Cambridge Philos. Soc. 146 (2009), no. 2, 407-413.

\bibitem[EG]{EG} S. Eilenberg, T. Ganea, {\em On the Lusternik-Schnirelmann Category of Abstract Groups.} Annals of
Math., 65, (1957), 517-518.

\bibitem[F]{F} R.H. Fox, On the Lusternik-Schnirelmann category, Annals of Math. 42 (1941), 33-370.


\bibitem[Gr]{Gr} M. Grant, 
\newblock    { https://mathoverflow.net/questions/89178/cohomological-dimension-of-a-homomorphism}

\bibitem[Ha]{Ha} A. Hatcher, {\em Algebraic Topology}, Cambridge University Press, 2002.

\bibitem[KLL]{KLL}  Kropholler, Peter; Linnell, Peter; Lück, Wolfgang Groups of small homological dimension and the Atiyah conjecture. Geometric and cohomological methods in group theory, 272-277, London Math. Soc. Lecture Note Ser., 358, Cambridge Univ. Press, Cambridge, 2009.

\bibitem[Mc]{Mc} D.R. McMillan, Cartesian products of contractible open manifolds,  Bull. Amer. Math. Soc. 67 (1961), 510-514.

\bibitem[McZ]{McZ} D.R. McMillan, E. C. Zeeman,  On contractible open manifolds. Proc. Cambridge Philos. Soc. 58 (1962), 221-224.

\bibitem[Sch]{Sch}  A.~Schwarz, The genus of a fibered space. Trudy Moscov. Mat. Obsc. 10, 11 (1961 and 1962), 217-272, 99-126.

\bibitem[Sc]{Sc} Jamie Scott,  On the topological complexity of maps. Topology Appl. 314 (2022), Paper No. 108094, 25 pp.,
 preprint arXiv:2011.10646 2020.

\bibitem[St]{St} J. Stallings, The piecewise-linear structure of Euclidean space. Proc. Cambridge Philos. Soc. 58 (1962), 481-488.

\end{thebibliography}
\end{document}